\documentclass[a4paper,11pt]{amsart}

\usepackage{amsmath}
\usepackage{mathtools}
\usepackage[all]{xy}
\usepackage[latin1]{inputenc}
\usepackage{varioref}
\usepackage{amsfonts}
\usepackage{amssymb}
\usepackage{bbm}
\usepackage{graphicx}
\usepackage{mathrsfs}
\usepackage[hypertexnames=false,backref=page,pdftex,
 	pdfpagemode=UseNone,
 	breaklinks=true,
 	extension=pdf,
 	colorlinks=true,
 	linkcolor=blue,
 	citecolor=red,
 	urlcolor=blue,
 ]{hyperref}

\newcommand{\sO}{{\mathcal O}}

\newcommand{\sX}{{\mathcal X}}

\newcommand{\scrD}{{\mathscr D}}

\newcommand{\scrH}{{\mathscr H}}

\newcommand{\scrU}{{\mathscr U}}

\newcommand{\scrX}{{\mathscr X}}

\newcommand{\C}{{\mathbb C}}

\newcommand{\HH}{{\mathbb H}}

\newcommand{\N}{{\mathbb N}}

\newcommand{\Q}{{\mathbb Q}}
\newcommand{\R}{{\mathbb R}}

\newcommand{\codim}{\operatorname{codim}}

\newcommand{\Ext}{\operatorname{Ext}}

\newcommand{\isom}{{\ \cong\ }}

\newcommand{\lt}{{\rm{lt}}}

\newcommand{\ohne}{{\ \setminus \ }}

\newcommand{\ratl}{\dashrightarrow}

\newcommand{\reg}{{\operatorname{reg}}}

\newcommand{\sing}{{\operatorname{sing}}}

\renewcommand{\to}[1][]{\xrightarrow{\ #1\ }}

\newcommand{\vphi}{\varphi}

\newcommand{\ul}[1]{{\underline{#1}}}

\newcommand{\wt}[1]{{\widetilde{#1}}}

\newtheoremstyle{citing}
  {}
  {}
  {\itshape}
  {}
  {\bfseries}
  {\textbf{.}}
  {.5em}
  {\thmnote{#3}}

\theoremstyle{plain}

\newtheorem{theorem}{Theorem}

\theoremstyle{definition}
\newtheorem{conjecture}[theorem]{Conjecture}
\newtheorem{corollary}[subsection]{Corollary}

\newtheorem{lemma}[theorem]{Lemma}
\newtheorem{proposition}[subsection]{Proposition}
\numberwithin{equation}{section}

\theoremstyle{remark}
\newtheorem{remark}[subsection]{Remark}

{\theoremstyle{citing}
\newtheorem*{custom}{}}

\newcommand{\coeff}{{\operatorname{coeff}}}
\newcommand{\mld}{{\operatorname{mld}}}

\newcommand{\tX}{{\tilde X}}
\newcommand{\Def}{\operatorname{Def}}

\newcommand{\rat}[1]{{\stackrel{#1}{\ratl}}}

\title[LMMP for symplectic varieties]{On the log minimal model program for irreducible symplectic varieties}

\author{Christian Lehn}
\address{Christian Lehn\\Institut de Math\'ematiques de Jussieu\\ 
4 place Jussieu\\Bo\^ite 247, 75252 Paris Cedex 05\\France}
\email{christian.lehn@imj-prg.fr}

\author{Gianluca Pacienza}
\address{Gianluca Pacienza\\Institut de Recherche Math\'ematique
Avanc\'ee\\ Universit\'e de Strasbourg et CNRS\\
7 rue Ren\'e Descartes\\67084 Strasbourg Cedex\\France}
\email{pacienza@math.unistra.fr}

\let\origmaketitle\maketitle
\def\maketitle{
  \begingroup
  \let\MakeUppercase\relax 
  \origmaketitle
  \endgroup
}

\numberwithin{theorem}{section}

\begin{document}
\thispagestyle{empty}

\begin{abstract}
We show the termination of any log-minimal model program for a pair $(X,\Delta)$ of a symplectic manifold $X$ and an effective $\R$-divisor $\Delta$.
\end{abstract}

\maketitle

\setlength{\parindent}{0em}
\setcounter{tocdepth}{1}



\section{Introduction}\label{sec intro}
\thispagestyle{empty}

The minimal model program is by today well-established and an indispensable tool in the study of higher dimensional varieties. One of its most important goals is to find good representatives, so called \emph{minimal models}, in every birational equivalence class of algebraic varieties, but also to determine, given a variety $X$, how to connect it to one of its minimal models by elementary birational transformations. 

Though there has been a lot of progress, these goals have not yet been completely accomplished. The existence of minimal models as well as the termination of certain special MMPs have been established in many cases in the seminal paper \cite{BCHM}. What is missing in general is termination of flips.


The goal of this paper is to show that symplectic manifolds behave as good as possible with respect to the MMP. For these manifolds, termination of log-flips has been shown by Matsushita \cite{Mat12} following a strategy due to Shokurov \cite{Sh03} (see below). However, the termination of log-flips does not imply that every MMP on a symplectic manifold terminates, for smoothness plays a crucial role in Matsushita's argument. For example, if the MMP produces not only flips but also divisorial contractions, the resulting variety will acquire singularities and then there could still be an infinite sequence of flips. Our main result is that this does not happen.

\begin{custom}[Theorem \ref{thm main}]
Let $X$ be a projective irreducible symplectic manifold and let $\Delta$ be an effective $\R$-Cartier divisor on $X$, such that the pair $(X,\Delta)$ is log-canonical. Then 
every log-MMP for $(X,\Delta)$ terminates in a minimal model $(X',\Delta')$ where $X'$ is a symplectic variety with canonical singularities  and $\Delta'$ is an effective, nef $\R$-Cartier divisor.
\end{custom}
It is well-known that from the previous result one derives the following (see \cite{Birkar} for the relevant definitions and further developments). 
\begin{corollary}
Let $X$ be a projective irreducible symplectic manifold and let $\Delta$ be an effective $\R$-Cartier divisor on $X$. Then birationally $\Delta$ has a Zariski decomposition in the sense of Fujita and in the sense of Cutkosky-Kawamata-Moriwaki.
\end{corollary}
The theorem above improves on \cite[Theorem 1.2]{Mat-Zh} where the authors prove the existence of log-minimal models for effective {\it movable} $\mathbb R$-divisors on an irreducible projective symplectic manifold.

The proof of Theorem \ref{thm main} follows Shokurov's strategy. Let us go a little more into details. To show the termination of flips, Shokurov introduced the so-called \emph{minimal log discrepancy} (mld for short), which is a local invariant associated with $(X,\Delta)$ and which increases under flips. It is nowadays interpreted as an invariant of the singularity of $(X,\Delta)$ at a given point. Shokurov has made two strong conjectures about the behaviour of mlds. These are the lower semi-continuity conjecture (LSC) and the ascending chain condition conjecture (ACC), see paragraph \ref{subsec mld conjectures}. Shokurov proved that these two conjectures imply termination of flips \cite{Sh03}. For smooth varieties, LSC holds by the fascinating paper \cite{EMY} and if all varieties in a sequence of flips are smooth, ACC holds for trivial reasons. However, even if we start with a smooth variety $X$, the MMP easily carries us out of the class of smooth varieties. Matsushita's keypoint is simply that a flip of a smooth symplectic variety remains smooth by deep results of Namikawa \cite{Na06}, see section \ref{sec term} for more details. 

Shokurov's conjectures ACC and LSC seem to be out of reach for arbitrary varieties. Starting with some variety $X$ and running a MMP might a priori produce a huge variety of different singularities. Nevertheless if we can bound the class of varieties that show up in intermediate steps of the MMP, then there is hope that Shokurov's strategy can be used.
In our case, this class of varieties will be the class of proper varieties with symplectic singularities in the sense of Beauville \cite{Be} which have a crepant resolution by a non-singular irreducible symplectic variety. We first prove the following result.

\begin{custom}[Theorem \ref{thm lsc}]
Let $Y$ be a normal projective $\Q$-Gorenstein variety and let $\Delta$ be an effective $\R$-Cartier divisor on $Y$ such that $(Y,\Delta)$ is log-canonical. If $\pi:X\to Y$ is a crepant morphism and LSC holds on $X$, then
LSC holds for $(Y,\Delta)$.
\end{custom}

In order to tackle ACC, we show that in a sequence of flips of singular symplectic varieties the singularities in fact \emph{do not change}. This allows us to use a result of Kawakita \cite{Kaw} to conclude. To the best of our knowledge, this strategy has been exploited for the first time by Nakamura \cite{Nak}, who considered the case of terminal quotient singularities. 
More concretely, we prove the following analogue of Huybrecht's theorem, which may be interesting on its own.

\begin{custom}[Theorem \ref{thm huybrechts}]
Let $X$ and $X'$ be projective symplectic varieties having crepant resolutions by irreducible symplectic manifolds and suppose that $\phi:X\ratl X'$ is a birational map which is an isomorphism in codimension $1$. Then $X$ and $X'$ are locally trivial deformations of one another.
\end{custom}

It is crucial for us that the deformations are locally trivial, as we want to relate the singularities of $X$ to those of $X'$. As far as we know, a singular version (with hypotheses different from ours) of Huybrechts' theorem has been proved and used in \cite[Proposition 4.2]{DV}, but it does not provide the local triviality of the deformations needed for our purposes.

\subsection*{Acknowledgments.} It is a pleasure to thank S\'ebastien Boucksom, St\'ephane Druel, James M\textsuperscript{c}Kernan and Yusuke Nakamura for helpful discussions. We thank Florin Ambro and Caucher Birkar for answering questions by e-mail.

C.L. was supported by the DFG through the research grant Le 3093/1-1 and enjoyed the hospitality of the IMJ, Paris.
G.P. was partially supported by the University of Strasbourg Institute for Advanced Study (USIAS), as part of a USIAS Fellowship, and by the 
 ANR project "CLASS'' no. ANR-10-JCJC-0111.

\section{Minimal log discrepancies}\label{sec mld}

\subsection{Notations and conventions}\label{notations}
A {\it log pair} $(X,\Delta)$ consists of a normal variety $X$ and a $\mathbb R$-Weil divisor $\Delta\geq 0$ such that $K_X+\Delta$ is $\mathbb R$-Cartier. 

A {\it log resolution} of a log pair $(X,\Delta)$ is a projective birational morphism 
$\pi: \tilde X\to X$ such that $\tilde X$ is smooth and $\pi^* \Delta + Exc(\pi)$ has simple normal crossing support.

A birational morphism $f:\tilde X\to X$ between varieties for which $K_X$ and $K_{\tilde X}$ are well-defined is called {\it crepant} if $\pi^*K_X=K_{\tilde X}$.
A {\it crepant resolution} is a resolution of singularities which is also a crepant morphism. 

A {\it symplectic variety} is a normal projective variety $X$ admitting an everywhere non-degenerate closed 2-form $\omega$ on the regular locus $X_{reg}$ of $X$ such that, for any resolution $f:\tilde X \to X$ with $f^{-1}(X_{reg})\cong X_{reg}$ the $2$-form $\omega$ extends to a regular and closed 2-form on $\tilde X$.

\subsection{Elementary properties of mlds}\label{subsec mld properties}
If $(X,\Delta)$ is a log pair and $\pi:\tilde X\to X$ is a log-resolution of $(X,\Delta)$, then we define the log discrepancy $a(E,X,\Delta)$ for a divisor $E$ over $X$ by the formula
\[
K_{\tilde X} + \tilde \Delta = \pi^* (K_{X} +  \Delta) + \sum_{E \subset \tilde X}  (1-a(E,X,\Delta)) E,
\]
where $\tilde \Delta$ is the strict transform of $\Delta$.

Let $c_X(E)\in X$ be the center of a divisor over $X$. This is a not necessarily closed point of $X$. The \emph{minimal log discrepancy} at $x\in X$ is
\[
\mld(x,X,\Delta):=\inf_{c_X(E)=x}a(E,X,\Delta)
\]
and the minimal log discrepancy along a subvariety $Z\subset X$ is
\[
\mld(Z,X,\Delta):=\inf_{x\in Z}\mld(x,X,\Delta).
\]

Notice that from the definition we have that 
\begin{equation}\label{eq:reverse}
Z\subset Z' \Rightarrow \mld(Z,X,\Delta) \geq \mld(Z',X,\Delta).
\end{equation}
Frequently we will write $\mld(x)$ and $\mld(Z)$ if there is no danger of confusion.
We refer to \cite[\S 1]{Am99} for more details.

We collect some basic facts about mlds.
\begin{lemma}\label{lemma mld under pullback}
Let $f:X\to Y$ be a proper birational morphism with $X$ normal and $\Q$-Gorenstein. Then
\[
\mld(W,Y,D) = \mld(\pi^{-1}(W), X, \pi^*D-K_{X/Y}).
\]
\end{lemma}
\begin{proof}
This is \cite[Proposition 1.3 (iv)]{EMY}.
\end{proof}
For $k\in \N$ let us denote by $X^{(k)} \subset X$ the subset of points of dimension $k$.  endowed with the subspace topology. The dimension of a point $x\in X$ is defined to be the dimension of the Zariski closure of $x$.
\begin{lemma}\label{lemma mld finite}
The function $\mld:= \mld_{(X,\Delta)}: X^{(k)}\to \R\cup\{-\infty\}$ takes only finitely many values.
\end{lemma}
\begin{proof}
This is \cite[Theorem 2.3]{Am99}.
\end{proof}

\subsection{Conjectures about mlds}\label{subsec mld conjectures}
Shokurov has made the two following conjectures about mlds in \cite{Sh03}.
\begin{conjecture}\label{hypo acc}
{\bf{(ACC)}} Let $\Gamma\subset [0,1]$ be a DCC-set, {\it i.e.}  all decreasing sequences in $\Gamma$ become eventually constant. For a fixed integer $k$ the set
\[
\Omega_k:= \left\{ \mld(Z,X,\Delta)\middle|
\begin{array}{l}
\dim X = k,\\ 
(X,\Delta) \textrm{ log pair}\\
Z\subset X \textrm{ closed subvariety}\\
\coeff(\Delta)\in \Gamma
\end{array}
\right\}
\]
is an ACC-set, that is, every increasing sequence $\alpha_1 \leq \alpha_2 \leq \ldots$ in $\Omega_k$ eventually becomes stationary.
\end{conjecture}
\begin{conjecture}\label{hypo lsc}
{\bf{(LSC)}} Let $X$ be a normal $\Q$-Gorenstein variety and let $\Delta$ be an $\R$-Weil divisor on $X$ such that $K_X+\Delta$ is $\R$-Cartier. Then for each $d$ the function $\mld_{(X,\Delta)}: X^{(d)} \to \R \cup \{-\infty\}$ is lower semi-continuous.
\end{conjecture}
\begin{remark}\label{remark lsc} If LSC holds on $X$, then for each $a\in \R$ and $d\in\N$ the set
\begin{equation}\label{eq lsc implies}
X_{\leq a}^{(d)} :=\{x \in X^{(d)}\mid \mld(x)\leq a\}
\end{equation}
is closed, that is, there is, $X_{\leq a}^{(d)} = X_{\leq a} \cap X^{(d)}$ where $X_{\leq a}$ is the closure of $X_{\leq a}^{(d)}$ in $X$. Moreover, $X_{a}^{(d)} :=\{x \in X^{(d)}\mid \mld(x)= a\}$ is open in $X_{\leq a}^{(d)}$. All this follows directly from Lemma \ref{lemma mld finite} which together with the lower semi-continuity implies that for $x\in X^{(d)}$ there is an open neighbourhood $U \subset X^{(d)}$ of $x$ such that
\[
\forall x'\in U: \ \mld(x)\leq \mld(x').
\]
It is well-known and an easy consequence of \cite[Prop. 2.5]{Am99} that lower semi-continuity is equivalent to $\mld: X^{(0)}\to \R\cup\{-\infty\}$ being lower semi-continuous. Moreover, by loc. cit. one also sees that ACC holds, as soon as
\[
\Omega_k^{(0)}:= \left\{ \mld(x,X,\Delta)\middle|
\begin{array}{l}
\dim X = k,\\ 
K_X+\Delta \ \Q\textrm{-Cartier},\\
x\in X \textrm{ closed point}\\
\coeff(\Delta)\in \Gamma
\end{array}
\right\}
\]
is an ACC set.
\end{remark}
Next we show that LSC descends along crepant morphisms. 
\begin{theorem}\label{thm lsc}
Let $Y$ be a normal projective $\Q$-Gorenstein variety and let $\Delta$ be an effective $\R$-Cartier divisor on $Y$ such that $(Y,\Delta)$ is log-canonical. If $\pi:X\to Y$ is a proper, crepant morphism and LSC holds on $X$, then
\[
\mld: Y^{(0)} \to \R \cup \{\infty\}
\]
is lower semi-continuous.
\end{theorem}
\begin{proof}
Let us fix a closed point $y\in Y$ and denote $W:=\pi^{-1}(y)$. By Lemma \ref{lemma mld under pullback} we have
\begin{equation}\label{eq mld}
\mld(y,Y,D)=\mld(\pi^{-1}(y),X,\pi^*D).
\end{equation}
We have to show that there is an open neighbourhood $U\subset Y$ of $y$ such that
\[
\mld(y') \geq \mld(y)\quad \forall y'\in U.
\]
To this end we spot the ``bad'' subsets of $Y$. By Lemma \ref{lemma mld finite} the function $\mld$ takes only finitely many values. If $a:=\mld(y)$ is the smallest $\mld$ on $Y^{(0)}$, then there is nothing to prove. Otherwise let us denote by $b$ the maximal $\mld$ on $Y$ with $b<a$. In view of \eqref{eq mld}, the search for mlds smaller than $a$ can be carried out on $X$, but at the price of having to take into account not only closed points. Consider for each $0\leq d \leq n:=\dim X$ the set
\[
C_d:= \{x \in X^{(d)}\mid \dim \pi(x) = 0, \mld(x)\leq b\}.
\]
Let $\overline {C}_d$ denote the Zariski closure of $C_d$ in $X$. By assumption, (LSC) holds on $X$ and hence all $x\in \overline {C}_d$ with $\dim x = d$ satisfy $\mld(x) \leq b$. Now we set
\[
U:=Y\ohne \bigcup_{d=0}^n\pi(\overline {C}_d),
\]
where $n=\dim (X)$.
As $\pi$ is proper, $U$ is open. We will consecutively prove the following claims.
\begin{enumerate}
	\item Every irreducible component of $\overline {C}_d$ has relative dimension at least $d$ over its image.
	\item $y\in U$.
	\item $\mld(y')\geq \mld(y)$ for all $y'\in U$.
\end{enumerate}
Let $\Sigma$ denote an irreducible component of $\overline {C}_d$ for some $d$. As $\overline {C}_d$ is the closure of $C_d$, the set $C_d\cap \Sigma$ is not empty. Thus, the set $\Sigma_{\geq d}:=\{x\in \Sigma\mid \dim \pi^{-1}\pi(x) \geq d \}$ is not empty. By the upper semi-continuity of the fiber dimension \cite[Corollaire 13.1.5]{EGAIV3},  $\Sigma_{\geq d}$ is closed and by definition we have $\Sigma_{\geq d} \supset \Sigma \cap C_d$. Therefore, as $\Sigma$ is a component of the closure of $C_d$ we have $\Sigma_{\geq d} = \Sigma$ and the first claim follows.

Suppose that $y\not \in U$. Then we would have a point $x \in \overline {C}_d$ for some $d$ with $\pi(x)=y$. By the previous statement, we have $\dim (W\cap \overline {C}_d) \geq d$, where, we recall, $W=\pi^{-1}(y)$. This implies that there is $x'\in W\cap \overline {C}_d$ with $\dim x' = d$ and hence $\mld(x') \leq b$ by the definition of $B_d$. But then, by (\ref{eq:reverse})
\[
a = \mld(y) \leq \mld(x') \leq b
\]
contradicting the choice of $b < a$. Thus $y\in U$.

Now if there were some $y'\in U$ with $\mld(y')< \mld(y)=a$ it would also be $\leq b$ by the maximality of $b$. Let $x \in \pi^{-1}(y')$ be a point with $\mld(x)= \mld(y')$ and denote $d:=\dim(x)$. This would imply $x\in C_d \subset \overline {C}_d$ contrarily to the assumption $y'\in U$. This concludes the proof of the theorem.
\end{proof}
By \cite[Thm. 0.3]{EMY}, LSC holds on smooth varieties. This immediately yields
\begin{corollary}\label{cor lsc}
Let $Y$ be a normal projective $\Q$-Gorenstein variety  possessing a crepant resolution of singularities. Let $\Delta$ be an $\R$-Weil divisor on $Y$ such that $K_Y+\Delta$ is $\R$-Cartier.  Then the function $\mld_{(Y,\Delta)}$ is lower semi-continuous.
\qed
\end{corollary}

\section{Deformations}\label{sec defo}
We work over the field of complex numbers. Recall that an irreducible symplectic manifold is by definition a compact simply connected K\"ahler manifold $X$ such that $H^0(X,\Omega_X^2)$ is generated by a holomorphic symplectic form.
In this section we are going to prove the following result which should be interpreted as an analogue of the well-known result of Huybrechts \cite[Thm. 2.5]{Huy}. The proof relies on  Namikawa's work \cite{Na01,Na06,Na10}. This is the only section where we make use of the complex numbers. This however does not seem to happen in an essential way and results as well as proofs should carry over mutatis mutandis to any algebraically closed field of characteristic zero.

\begin{theorem}\label{thm huybrechts}
Let $X$ and $X'$ be projective symplectic varieties having crepant resolutions by irreducible symplectic manifolds and suppose that $\phi:X\ratl X'$ is a birational map which is an isomorphism in codimension $1$. Then $X$ and $X'$ are locally trivial deformations of one another.
\end{theorem}

The essential ingredient is the smoothness of certain deformation spaces, which is obtained by invoking Ran's $T^1$-lifting principle \cite{Ran,Ka1,Ka2}. Essentially it says that a given deformation problem is unobstructed if the the tangent space $T^1_X$ to the deformation space is ``deformation invariant'' in the sense that its relative version $T^1_{\sX/S}$ is a free $\sO_S$-module for every small deformation $\sX\to S$ of $X$. We refer to \cite[4.13]{L} or  \cite[\S 14]{GHJ} for a concise account. Recall (e.g. from \cite[1.2]{Ser}) that the tangent space to the deformation functor $\Def^\lt_X$ of locally trivial deformations of an algebraic variety $X$ is $H^1(T_X)$, opposed to arbitrary deformations, where the tangent space is $\Ext^1(\Omega_X,\sO_X)$. An obstruction space for $\Def^\lt_X$ is given by $H^2(T_X)$.

Let us recall the following well-known result on the local structure of singular symplectic varieties. For convenience we sketch the proof which is due to Kaledin and Namikawa, see \cite{Na10}.

\begin{proposition}\label{proposition kaledin}
Let $X$ be a symplectic variety. Then there is an open subset $U\subset X$ such that $\codim_X (X\ohne U) \geq 4$ and every $x\in U$ has a neighbourhood which is locally analytically isomorphic to $(\C^{2n-2},0)\times (S,p)$ where $2n=\dim X$ and $(S,p)$ is the germ of a rational double point on a surface. This isomorphism can be chosen to preserve the symplectic structure.
\end{proposition}
\begin{proof}
Let $\Sigma \subset X$ be the singular locus of the reduction of $X^\sing$. Kaledin's result \cite[Thm. 2.3]{Kal} implies that $\Sigma$ has codimension $\geq 4$ and that every point of $U:=X\ohne \Sigma$ admits the sought for product decomposition in the formal category. By \cite[Corollary 2.6]{Ar} the decomposition exists analytically. The last statement is \cite[Lemma 1.3]{Na10}.
\end{proof}

\begin{proposition}\label{proposition namikawa}
Let $X$ be a compact symplectic variety, let $\pi:\tilde X\to X$ be a crepant resolution by a compact K\"ahler manifold $\tilde X$ and let $U\subset X$ be as in Proposition \ref{proposition kaledin}. Then the restriction $H^1(X,T_X) \to H^1(U,T_U)$ is an isomorphism and $h^1(T_X)=h^1(T_{\tilde X})-m$ where $m$ is the number of irreducible components of the exceptional divisor of $\pi$.
\end{proposition}
\begin{proof}
Let us consider the diagram
\[
\xymatrix{
0 \ar[r] & H^1(T_{X}) \ar[r]\ar[d] & Ext^1(\Omega_X,\sO_X) \ar[r]\ar[d]^{\phi} & H^0(T_X^1) \ar[d]\\
0 \ar[r] & H^1(T_U) \ar[r] & Ext^1(\Omega_U,\sO_U) \ar[r] & H^0(T_U^1) \ar[r] & 0\\
}
\]
with exact lines, where exactness at $H^0(T_U^1)$ is shown in (ii) of the proof of \cite[Theorem 2.2]{Na01}. Moreover, $\phi$ is an isomorphism by \cite[Proposition 2.1]{Na01}. In order to show that $H^1(T_X) \to H^1(T_U)$ is an isomorphism it hence suffices to show that 
\begin{equation}\label{eq h1tx}
h^1(T_X) \geq h^1(T_U). 
\end{equation}
 Let us consider the following diagram. Here $\tilde U \subset \tilde X$ denotes the preimage of $U$ under $\pi$. 

\begin{equation}\label{eq diag restr}
\xymatrix{
0 \ar[r] & H^1(\pi_* T_{\tilde X}) \ar[r]\ar[d] & H^1(T_{\tilde X}) \ar[r]\ar[d]^{\phi'} & H^0(R^1\pi_*T_{\tilde X})\ar[d]^\psi\\
0 \ar[r] & H^1(\pi_* T_{\tilde U}) \ar[r] & H^1(T_{\tilde U}) \ar[r] & H^0(R^1\pi_*T_{\tilde U})\ar[r] & 0\\
}
\end{equation}
As $T_X$ is reflexive, we have $j_*T_{U} = T_X$ where $j:U\to X$ denotes the inclusion and hence $\pi_* T_{\tilde X}=T_X$. So $H^1(\pi_* T_{\tilde X})$ has dimension $h^1(T_X)$. The same argument shows that $H^1(\pi_* T_{\tilde U})$ has dimension $h^1(T_U)$. Again by \cite[Proposition 2.1]{Na01}, the morphism $\phi'$ is an isomorphism. 
Moreover, exactness at $H^0(R^1\pi_*T_{\tilde U})$ and the equality $h^1(T_U)=h^1(T_{\tilde U})-m$ have been shown in (ii) of the proof of \cite[Theorem 2.2]{Na01}.
So the desired inequality \eqref{eq h1tx} follows if in diagram \eqref{eq diag restr} the morphism $\psi$ is injective. But already the morphism $R^1\pi_*T_{\tilde X} \to R^1\pi_*T_{\tilde U}$ is injective which follows from $\scrH^1_{\tilde Z}(T_{\tilde X})=0$ where $\tilde Z = \tilde X \ohne \tilde U$ because $\tilde Z$ has codimension $\geq 2$ and $T_{\tilde X}$ is locally free. This concludes the proof.
\end{proof}

In the situation of Theorem \ref{thm huybrechts}, let $\pi: \tX\to X$ be a crepant resolution and let $D=\sum_{i=1}^mD_i$ be the exceptional divisor with its decomposition into irreducible components $D_i$. Denote by $L_i:=\sO_\tX(D_i)$.
We denote by $\tilde\scrX\to \Def(\tX)$ the universal deformation of $\tX$. This is the germ of a smooth space of dimension $h^{1,1}(\tX)$ by the Bogomolov-Tian-Todorov theorem. We consider the following subspaces of $\Def(\tX)$:

\begin{itemize}
\item $\Def(\tX,\ul L) \subset \Def(\tX)$ is the base of the universal deformation of $(\tX,L_1,\ldots,L_m)$, see \cite[(1.14)]{Huy99}. As the $D_i$ define linearly independent classes in $H^2(\tX,\C)$, $\Def(\tX,\ul L)$ is smooth and of codimension $m$ in $\Def(\tX)$ by loc. cit.
\item $\Def(\tX,\ul D) \subset \Def(\tX)$ is the image of the components containing all $D_i$ of the relative Douady space $\scrD(\tilde\scrX/\Def(\tX))\to \Def(\tX)$. This is the space where all components $D_i$ deform along with $\tX$.
\end{itemize}

We clearly have $\Def(\tX,\ul D) \subset \Def(\tX,\ul L)$
Consequently, $\dim \Def(\tX,\ul D)\leq h^{1,1}(\tX)-m$.

The key step will be to prove the smoothness of the space of locally trivial deformations of the singular variety $X$. 

\begin{proposition}\label{prop defo}
Let $\pi: \tX\to X$ be as above. Let $\tilde\scrX \to \Def(\tX,\ul D)$ and $\scrX \to\Def^\lt(X)$ be the universal deformations. Then there is a diagram 
\begin{equation}\label{eq diag lt}
\xymatrix{
\tilde\scrX \ar[d]\ar[r]& \scrX \ar[d]\\
\Def(\tX,\ul D) \ar[r]^{\pi_*} & \Def^\lt(X) \\
}
\end{equation}
with the following properties:
\begin{enumerate}
\item\label{item eins} $\Def^\lt(X)$ is smooth of dimension $h^{1,1}(\tX)-m$.
\item\label{item zwei} $\pi_*$ is the restriction of the natural finite morphism $\Pi:\Def(\tX)\to\Def(X)$. 
\item\label{item drei} $\dim \Def(\tX,\ul D) = h^{1,1}(\tX)-m$, in particular $\Def(\tX,\ul D)=\Def(\tX,\ul L)$.
\end{enumerate}
\end{proposition}
\begin{proof}
We will first show that $\Def^\lt(X)$ is smooth.
Let $U\subset X$ be as in Proposition \ref{proposition kaledin}.
The restriction $H^1(T_X) \to H^1(T_U)$ is an isomorphism by Proposition \ref{proposition namikawa}, in other words, deformations and their local triviality are determined on $U$. Let $j:X^\reg \to U$ denote the inclusion. As $T_U$ is reflexive, we have that $T_U\isom j_* T_{X^\reg}$. 
Hence $H^1(T_X)=H^1(U,j_*\Omega_{X^\reg})=\HH^2(U,j_*\Omega_{X^\reg}^{\geq 1})$ which is deformation invariant. To see this last claim we consider the exact sequence of complexes
\begin{equation}\label{eq komplexe}
0\to j_*\Omega_{X^\reg}^{\geq 1} \to j_*\Omega_{X^\reg}^{\bullet} \to \sO_U \to 0.
\end{equation}
By Grothendieck's theorem for V-manifolds $\HH^k(j_*\Omega_{X^\reg}^{\bullet}) = H^k(U,\C)$, see for example the footnote in the proof of \cite[Proposition (1.11)]{Na06}. Moreover, we have
\[
H^1(\sO_U) = H^1(\sO_X) = H^1(\sO_{\tilde X}) = 0,
\]
where the first equality holds because $X$ is Cohen-Macaulay and $\codim_X (X\ohne U) \geq 4$ and the second because $X$ has rational singularities. In the same way, one finds
\[
H^2(\sO_U) = H^2(\sO_X) = H^2(\sO_{\tilde X}) \isom \C.
\]

so that \eqref{eq komplexe} gives an exact sequence
\[
0\to H^1(j_*\Omega_{X^\reg}) \to H^2(U,\C) \to H^2(\sO_U) \to 0,
\]
where the last map is surjective because the composition $H^2(\tilde X,\C) \to H^2(\sO_{\tilde X}) \to[\isom] H^2(\sO_U) $ is.
The same line of arguments works identically in a relative situation and shows that $H^1(T_{\sX/S})=H^1(j_*\Omega_{(\sX/S)^\reg})$ is a free $\sO_S$-module for any small deformation $\sX\to S$ over a local artinian scheme $S$. In other words, the tangent space to the deformation functor $H^1(T_X)$ is deformation invariant, hence by the $T^1$-lifting argument $\Def^\lt(X)$ is smooth. From Proposition \ref{proposition namikawa} it follows that $\dim \Def^\lt(X)=h^{1,1}(\tX)-m$.


As explained in \cite[\S 3]{Na06} there is a diagram as \eqref{eq diag lt} for arbitrary instead of locally trivial deformations.  In particular, there is a finite map $\Pi:\Def(\tX)\to\Def(X)$ and for each $t\in \Def^\lt(X)$ and every $s\in \Pi^{-1}(t)$ the morphism $\tilde\scrX_s\to \scrX_{t}$  is a crepant resolution.
Denote by  $\scrU\subset \scrX$ be the induced locally trivial deformation of $U\subset X$. Note that by the choice of $\scrU$ it is a locally trivial deformation of an ADE-surface singularity, thus it has a unique minimal relative resolution $\wt\scrU\to \scrU$. By uniqueness, $\wt\scrU_t$ embeds into $\tilde\scrX_s$ for each $t\in \Def^\lt(X)$ and $s\in \Pi^{-1}(t)$. Taking the closure of the irreducible components of the exceptional divisor of this morphism in $\wt\scrX$ we obtain deformations of the $D_i$ over $\Pi^{-1}(\Def^\lt(X))$ and thus the inclusion $\Pi^{-1}(\Def^\lt(X)) \subset \Def(\tX,\ul D)$ holds. As $\dim\Def(\tX,\ul D)\leq h^{1,1}(\tX)-m$, this inclusion is an equality and the restriction of $\Pi$ to $\Def(\tX,\ul D)$ gives the desired morphism. Also \eqref{item drei} follows from this.
\end{proof}

\begin{proposition}\label{prop defo birational}
Let $X\ratl X'$ be as in Theorem \ref{thm huybrechts} and let $\scrX \to \Def^\lt(X)$ and $\scrX' \to \Def^\lt(X')$ be the universal locally trivial deformations of $X$ respectively $X'$. Then there is a correspondence $\Gamma \subset \Def^\lt(X)\times \Def^\lt(X')$ surjecting onto each factor such that for each $(t,t') \in \Gamma$ we have a birational map $\phi_{(t,t')}:\scrX_t \ratl \scrX'_{t'}$. For general $(t,t')\in \Gamma$, the map $\phi_{(t,t')}$ is an isomorphism.
\end{proposition}
\begin{proof}
Let $\tX\to X$, $\tX'\to X'$ be crepant resolutions of singularities and consider the morphisms $\Def(\tX,\ul D) \to[\pi_*] \Def^\lt(X)$ and $\Def(\tX',\ul D) \to[\pi'_*] \Def^\lt(X')$ from Proposition \ref{prop defo}. As $\tX \ratl \tX'$ is an isomorphism in codimension $1$, the local Torelli theorem gives an isomorphism $\vphi:\Def(\tX,\ul D)\to \Def(\tX',\ul D)$ and under this identification the fibers $\tilde\scrX_t$ and $\tilde\scrX'_t$ will be birational. As the morphisms $\pi_t: \tilde\scrX_t \to \scrX_t$ and $\pi_t': \tilde\scrX'_t \to \scrX'_t$ contract the same divisors, we obtain a birational map $\scrX_{\pi_*(t)}\to \scrX'_{\pi'_*(t)}$ which is isomorphic in codimension one. We put $\Gamma:=\left(\pi_*\times (\pi'_*\circ \vphi)\right)(\Def(\tX,\ul D))$. The last statement follows as the general projective deformation of $X$ has Picard number one. Note that we can always deform to projective varieties.
\end{proof}

This concludes the proof of Theorem \ref{thm huybrechts}.\qed

\section{Termination}\label{sec term}
In this section we prove our main result:
\begin{theorem}\label{thm main}
Let $X$ be a projective irreducible symplectic manifold and let $\Delta$ be an effective $\R$-Cartier divisor on $X$, such that the pair $(X,\Delta)$ is log-canonical. Then log-MMPs for $(X,\Delta)$ exist and every log-MMP terminates in a minimal model $(X',\Delta')$ where $X'$ is a symplectic with canonical singularities variety and $\Delta'$ is an effective, nef $\R$-Cartier divisor.
\end{theorem}
In fact, we could also drop the lc-assumption on $(X,\Delta)$, as thanks to $K_X=0$ we can rescale $\Delta$ at any time. The proof of the theorem will occupy the rest of the section.
Let $(X,\Delta)$ be a log pair on a projective irreducible symplectic manifold. By \cite[Corollary 1.4.1]{BCHM}, $(K_X+\Delta)$-flips exist and so we may run a $(K_X+\Delta)$-MMP. 
This produces a sequence 
\begin{equation}\label{eq mmp}
X=X_0 \rat{\phi_0} X_1 \rat{\phi_1} X_2 \ratl \ldots 
\end{equation}
where the $\phi_i$ are either divisorial contractions or flips. Let us denote $\Delta_0:=\Delta$ and $\Delta_i = (\phi_i)_* \Delta_{i-1}$. Note that at each step $K_{X_i}$ will be trivial and therefore we can rescale $\Delta_i$ such that $(X_i,\Delta_i)$ will be klt, hence the above result applies.  We want to show that \eqref{eq mmp} terminates after a finite number of steps. First we notice:

\begin{lemma}\label{lemma crepant resolution}
Each $X_i$ is a symplectic variety and admits a crepant resolution.
\end{lemma}
\begin{proof}
By induction we may assume that $X_{i-1}$ is a symplectic variety and has a crepant resolution $\tilde\pi:\tX_{i-1}\to X_{i-1}$.  Symplecticity of $X_i$ is clear, as the exceptional locus of $X_{i-1} \ratl X_{i}$ on $X_{i}$ has codimension $\geq 2$ and thus the symplectic form from $X_{i-1}$ extends. By \cite[Corollary 1.4.3]{BCHM} there exists a proper birational morphism $\pi:\tX_i\to X_i$ such that $\tX_i$ has only $\Q$-factorial terminal singularities and $\pi$ is crepant. Let $X_{i-1}\to Z \leftarrow X_{i}$ be the flipping contraction. Then the compositions $\tX_i\to X_i\to Z$ and $\tX_{i-1}\to X_{i-1}\to Z$ are crepant morphisms and $\tX_{i-1}$ is smooth, hence by \cite[Corollary 1, p. 98]{Na06} also $\tX_{i}$ is smooth.
\end{proof}

In the course of the MMP, only a finite number of divisorial contractions can occur so that we can reduce to the following situation: $X=X_0$ is a symplectic variety having a crepant resolution, $\Delta=\Delta_0$ is an effective $\R$-divisor on $X$ and we are given a sequence
\begin{equation}\label{eq sequence of flips}
X=X_0 \rat{\phi_0} X_1 \rat{\phi_1} X_2 \ratl \ldots 
\end{equation}
where $\phi_i$ is a log-flip for the pair $(X_i,\Delta_i)$. We will show that such a sequence is finite by using Shokurov's criterion. We only need LSC for the pairs $(X_i,\Delta_i)$ and ACC for the set  
\[
\Omega(X):= \left\{ \mld(E_i,X_i,\Delta_i)\middle|
i\in \N\right\}
\]
where $E_i\subset X_i$ denotes the exceptional locus of $\phi_i:X_i\to X_{i+1}$, see also \cite[\S 3]{HM10}.
Recalling that LSC holds by Corollary \ref{cor lsc} and Lemma \ref{lemma crepant resolution} above, we are left with ACC. 
By a theorem of Kawakita \cite{Kaw}, the set of all mlds for a fixed finite set of coefficients on a fixed projective variety is finite. More precisely, let $X$ be a projective variety, let $\Gamma \subset [0,1]$ be a finite set and consider 
\[
M_\Gamma(X):=\{ \mld_{(X,\Delta)}(x) \mid \coeff(\Delta)\in \Gamma, (X,\Delta)\textrm{ lc at } x\in X^{(0)}\}
\]
Then by \cite[Thm. 1.2]{Kaw}, the set $M_\Gamma(X)$ is finite. By Theorem \ref{thm huybrechts} all $X_i$ in the sequence \eqref{eq sequence of flips} are locally trivial deformations of one another. Thus the ACC condition is satisfied by the following observation, which has already been made in \cite[Corollary 1.4]{Nak}.

\begin{proposition}
If $X$ and $X'$ are locally trivial deformations of one another, then $M_\Gamma(X)=M_\Gamma(X')$.
\end{proposition}
\begin{proof}
By definition of $M_\Gamma$ we only consider mlds that are attained in closed points. But then it follows from the fact that mlds are local invariants.
\end{proof}
As by Remark \ref{remark lsc} it suffices to have ACC in closed points, the hypotheses of Shokurov's criterion are fulfilled so that we may conclude the proof of Theorem \ref{thm main}. \qed

\end{document}